\documentclass{amsart}
\usepackage{amsfonts,amssymb,amsmath,amscd,amsthm,txfonts,graphicx,color,enumerate}
\usepackage{mathscinet}
\usepackage{calrsfs}
\usepackage{mathtools}
\usepackage{hyperref}

\newtheorem{theorem}{Theorem}
\newtheorem{proposition}{Proposition}%[theorem]
\newtheorem{corollary}{Corollary}
\newtheorem{lemma}{Lemma}

\newtheorem{remark}{Remark}
\newcommand{\N}{\mathbb{N}}
\newcommand{\R}{\mathbb{R}}

\renewcommand{\P}{\operatorname{\mathsf{P}}}
\newcommand{\E}{\operatorname{\mathsf{E}}}
\newcommand{\Var}{\operatorname{\mathsf{Var}}}

\newcommand{\hi}{\operatorname{\mathsf{sup}}}
\newcommand{\lo}{\operatorname{\mathsf{inf}}}

\renewcommand{\le}{\leqslant}
\renewcommand{\ge}{\geqslant}

\newcommand{\supp}{\operatorname{supp}}

\newcommand{\dd}{\operatorname{d}\!}

\newcommand{\be}{\beta}
\newcommand{\ga}{\gamma}
\newcommand{\vp}{\varepsilon}

\newcommand{\tX}{\tilde{X}}

\newcommand{\X}{\mathcal{X}}
\newcommand{\Y}{\mathcal{Y}}

\newcommand{\ol}{\overline}

\newcommand{\g}{\mathsf{g}}

\begin{document}
%\runningtitle{Bounds on the arithmetic and geometric means}
\title{Exact upper and lower bounds on the difference between the arithmetic and geometric means}
%% If there is more than one author, put \cauthor immediately before
%% the corresponding author.
%\cauthor %% mark the next author as corresponding author
\author{Iosif Pinelis}
\address{Department of Mathematical Sciences\\
Michigan Technological University\\
Houghton, Michigan 49931, USA}
\email{ipinelis@mtu.edu}
%% If there are several authors, list them here
%\author[2]{Second author}
%\address[2]{Second address\email{a@net.com}}

%% List the authors, initials and surnames only, for the
%% running head (left hand page)
%\authorheadline{I. Pinelis}

%% If there is a dedication, include it here
%\dedication{Dedicated to ...}

%\support{Include acknowledgement of support here}

\begin{abstract}
Let $X$ denote a nonnegative random variable with $\E X<\infty$. 
Upper and lower bounds on $\E X-\exp\E\ln X$ are obtained, which are exact, in terms of $V_X$ and $E_X$ for the upper bound and in terms of $V_X$ and $F_X$ for the lower bound, where $V_X:=\Var\sqrt X$, $E_X:=\E\big(\sqrt X-\sqrt{m_X}\,\big)^2$, $F_X:=\E\big(\sqrt{M_X}-\sqrt X\,\big)^2$, 
$m_X:=\inf S_X$, $M_X:=\sup S_X$, and $S_X$ is the support set of the distribution of $X$. 
Note that, if $X$ takes each of distinct real values $x_1,\dots,x_n$ with probability $1/n$, then $\E X$ and $\exp\E\ln X$ are, respectively, the arithmetic and geometric means of $x_1,\dots,x_n$.  
\end{abstract}

%% - subject classification and keywords
%% 2010 American Mathematical Society Subject Classification
%% Provide only ONE primary classification

% 	26D15   	Inequalities for sums, series and integrals
% 	90C46   	Optimality conditions, duality
%% Four or five keywords or phrases

\maketitle

\section{%Results
%Introduction, s
Summary and discussion}
Let $\X_+$ denote the set of all nonnegative random variables (r.v.'s) $X$ with $\E X<\infty$. 
Take any $X\in\X_+$ and let
\begin{alignat}{2}
	 V_X:=\Var\sqrt X, \quad 
	m_X:=&\inf\supp X, \quad M_X:=\sup\supp X, \notag \\ 
	\quad E_X:=&\E\big(\sqrt X-\sqrt{m_X}\,\big)^2, \quad 
	F_X:=\E\big(\sqrt{M_X}-\sqrt X\,\big)^2,%\quad\text{and}
	\label{eq:E,F,V}
\end{alignat}
where, as usual, $\supp X$ denotes the support of (the distribution of) the r.v.\ $X$. 

It will be shown in this note that 
\begin{equation}\label{eq:}
	(2V_X)\wedge\frac{F_X V_X}{F_X-V_X}\le\E X-\exp\E\ln X \le(2V_X)\vee E_X  
\end{equation}
and that each of these two bounds on $\E X-\exp\E\ln X$ is exact, in terms of $V_X$ and $E_X$ for the upper bound and in terms of $V_X$ and $F_X$ for the lower bound.  
%
%$V_X$, $E_X$, and $F_X$. 
As usual, for any real numbers $z_1,\dots,z_n$, we write $z_1\vee\dots\vee z_n$ and $z_1\wedge\dots\wedge z_n$ for their maximum and 
minimum, respectively.  

Since the r.v.\ $X$ is nonnegative, clearly $m_X\in[0,\infty)$. However, concerning the value of $M_X$ one can then only say that $M_X\in[m_X,\infty]$, with the case $M_X=\infty$ certainly possible. 
Next, given the condition $\E X<\infty$, the values of $E_X$ and $V_X$ are necessarily finite, and hence so is the upper bound in \eqref{eq:}. On the other hand, $F_X=\infty$ if $M_X=\infty$; however, even then, the lower bound in \eqref{eq:} will of course be finite. 
Concerning the ratio $\frac{F_X V_X}{F_X-V_X}$ in the lower bound in \eqref{eq:}, for any $V\in\R$, $E\in\R$, and $F\in(-\infty,\infty]$ we assume the conventions that $\frac{F V}{F-V}$ equals $V$ if $F=\infty$ and equals $0$ if $F=V$.  
It will be seen %soon 
that these conventions are the appropriate ones  in the present context. 

%More specifically, one has 

That the upper and lower bounds in \eqref{eq:} hold and are exact will be established in Theorem~\ref{th:} below. The statement of Theorem~\ref{th:} is preceded by three propositions, which complement and help understand %provide commentaries on 
the main result. 

%the following theorem; see especially its part~(III). 
%Let $\X_+$ denote the set of all nonnegative r.v.'s $X$ with $\E X<\infty$. 

Take any $V\in\R$, $E\in\R$, and $F\in(-\infty,\infty]$. 

Introduce the sets  	
\begin{align}
	\X_{\hi;V,E}:=&\big\{X\in\X_+\colon V_X=V, E_X=E\big\}, \label{eq:X_hi} \\ 
	\X_{\lo;V,F}:=&\big\{X\in\X_+\colon V_X=V, F_X=F\big\}. \label{eq:X_lo}
\end{align}

\begin{proposition}\label{prop:1}\ 
One has $\X_{\hi;V,E}\ne\emptyset$ if and only if 
\begin{equation}\label{eq:nonempty,hi}
\text{either $E=V=0$ or	$E>V>0$.} 	
\end{equation}
Similarly, $\X_{\lo;V,F}\ne\emptyset$ if and only if   
\begin{equation}\label{eq:nonempty,lo}
\text{either $F=V=0$ or	$F>V>0$.} 	
\end{equation}
\end{proposition}

All the necessary proofs are given in Section~\ref{proofs}. 

Values of $V$ and $E$ as in \eqref{eq:nonempty,hi}, as well as values of $V$ and $F$ as in \eqref{eq:nonempty,lo}, may be referred to as admissible. 

\begin{proposition}\label{prop:2}
If $\X_{\lo;V,F}\ne\emptyset$, then 
\begin{equation}\label{eq:E_{V,F}}
	E_{V,F}:=\frac{F V}{F-V}=\inf\big\{E_X\colon X\in\X_{\lo;V,F}\big\}.
\end{equation}
If, moreover, $F<\infty$, then the latter infimum is attained, and it is attained at a r.v.\ $X\in\X_{\lo;V,F}$ if and only if $\supp X=\{m_X,M_X\}$ -- that is, if and only if $\supp X$ contains at most two points. 
If $F=\infty$, then the infimum in \eqref{eq:E_{V,F}} is not attained. 	
\end{proposition}

\begin{proposition}\label{prop:3}
Take any $X\in\X_+$. Then both inequalities in \eqref{eq:} turn simultaneously into the equalities if and only if the distribution of the r.v.\ $\sqrt X$ is the symmetric distribution on a set of at most two points in $[0,\infty)$. 
\end{proposition}

\begin{theorem}\label{th:}\ 
%
%\begin{enumerate}[(I)]
%	\item\label{optim} 
%	Introducing 
%the notation  
Let 
\begin{align}
D_X:=&\E X-\exp\E\ln X.   \label{eq:D_X}  
\end{align}	
Then 	
\begin{alignat}{3}
S_{V,E}&:=\sup\big\{D_X\colon X\in\X_{\hi;V,E}\big\}=&&(2V)\vee E
	\ &&\text{if}\ \X_{\hi;V,E}\ne\emptyset; %\quad\text{and} 
	\label{eq:sup} \\ 
I_{V,F}&:=\inf\big\{D_X\colon X\in\X_{\lo;V,F}\big\}=&&(2V)\wedge E_{V,F}
		\quad&&\text{if}\ \X_{\lo;V,F}\ne\emptyset.  \label{eq:inf} 
\end{alignat}
These equalities hold if the sets $\X_{\hi;V,E}$ and $\X_{\hi;V,E}$ are replaced there by their respective subsets consisting of the r.v.'s in $\X_{\hi;V,E}$ and $\X_{\hi;V,E}$ taking at most two values. 
%
%\end{enumerate}
\end{theorem}

Clearly, inequalities \eqref{eq:} and the exactness of the upper and lower bounds in \eqref{eq:} immediately follow from Theorem~\ref{th:}. 

\begin{remark}\label{rem:le}
Note that $(2V)\vee E$ is nondecreasing in $V$ and $E$, whereas $(2V)\wedge E_{V,F}$ is nondecreasing in $V$ and nonincreasing in $F$ \big(from $E_{V,V+}=2V$ down to $E_{V,\infty}=V$\big). 
So, \eqref{eq:sup} will hold if the equalities $V_X=V$ and $E_X=E$ in the definition \eqref{eq:X_hi} of $\X_{\hi;V,E}$ are replaced by the inequalities $V_X\le V$ and $E_X\le E$.  
Similarly, \eqref{eq:inf} will hold if the equalities $V_X=V$ and $F_X=F$ in the definition of \eqref{eq:X_lo} of $\X_{\lo;V,F}$ are replaced by $V_X\ge V$ and $F_X\le F$.  

Moreover, it is now clear that inequalities \eqref{eq:} will hold if $m_X$ and $M_X$ in the definitions of $E_X$ and $F_X$ in \eqref{eq:E,F,V} are replaced, respectively, by any nonnegative $a$ and $b$ such that $\supp X\subseteq[a,b]$. 

It also follows from the mentioned monotonicity of the exact lower bound \break 
$(2V)\wedge E_{V,F}$ in $F$ that the values of this bound are always between $V$ and $2V$. \qed
\end{remark}

The lower bound in \eqref{eq:} is an improvement of the zero bound, which follows immediately by the Jensen inequality for the (convex) exponential function. 
% \big(or, alternatively, for the (concave) logarithmic function\big). 
In particular, the condition $\E X<\infty$ implies $\E\ln X<\infty$; however, it is possible that $\E\ln X=-\infty$;  
we use the standard conventions $\ln0:=-\infty$ and $\exp(-\infty):=0$. 

As for the second inequality in \eqref{eq:}, one may consider it 
as a reverse Jensen inequality; cf.\ e.g.\ \cite{dragomir13}. In contrast with the upper bound in \eqref{eq:}, the bounds in \cite{dragomir13} will be finite only when $M_X-m_X<\infty$. On the other hand, the bounds in \eqref{eq:} are only for the case when the convex function is the exponential one.

In the case when the r.v.\ $X$ is a continuous function on the interval $[0,1]$ endowed with the Lebesgue measure, obtaining the upper bound $(\sqrt{M_X}-\sqrt m_X)^2$ on \break 
$\E X-\exp\E\ln X$ 
%with $M_X:=\max_{t\in[0,1]}X(t)$ 
was presented as Problem~1180 in \cite{11800}. 
Note that $2V_X=2\Var\sqrt X$ can be rewritten as $\E\big(\sqrt X-\sqrt{\tX}\;\big)^2$, where $\tX$ is an independent copy of the r.v.\ $X$. Therefore, the upper bound in \eqref{eq:} is strictly less than that in \cite{11800} unless $\supp X=\{m_X,M_X\}$. 
%contains at most two distinct points. 
In the case when $X$ is a continuous function on the interval $[0,1]$, the latter condition on $\supp X$ simply means that $X$ is a constant, and then the difference $\E X-\exp\E\ln X$ and the upper bound on it in \eqref{eq:} (as well as the lower one) are each $0$.

%One may also note that, under the condition $\E X<\infty$, the upper and lower bounds on $\E X-\exp\E\ln X$ given in \eqref{eq:} are always finite. 

Given any nonnegative real numbers $x_1,\dots,x_n$, let $X$ be any r.v.\ with the distribution defined by the formula 
\begin{equation}\label{eq:discr}
	\E f(X)=\frac1n\,\sum_{i=1}^n f(x_i)\quad\text{for any function $f\colon\R\to\R$.}
\end{equation}
\big(So, in the case when the numbers $x_1,\dots,x_n$ are pairwise distinct, any such r.v.\ $X$ takes each of the values $x_1,\dots,x_n$ with probability $\frac1n$.\big) 
%
%In the case when the r.v.\ $X$ takes nonnegative real values $x_1,\dots,x_n$, each with probability $\frac1n$, \eqref{eq:} results in the following inequalities for the the difference between the arithmetic and geometric means: 
%\begin{equation}\label{eq:a-g}
%\begin{aligned}
%	0\le&\,\frac{x_1+\dots+x_n}n-\sqrt[\leftroot{3}\uproot{3}n]{x_1\cdots x_n} \\
%	\le&\max\Big[\frac2n \sum_1^n\big(\sqrt{x_i}-\ol{\sqrt x}\big)^2,\; 
%	\frac1n \sum_1^n\big(\sqrt{x_i}-\min_i\sqrt{x_i}\big)^2\,\Big], 
%\end{aligned}	
%\end{equation}
%where $\ol{\sqrt x}:=\frac1n \sum_1^n\sqrt{x_i}$. 
In this case, 
\begin{equation}\label{eq:a-g}
\E X=%\ol x:=
\frac{x_1+\dots+x_n}n\quad\text{and}\quad 
\exp\E\ln X= %x^\g:=
\sqrt[\leftroot{3}\uproot{3}n]{x_1\cdots x_n}. 
\end{equation}
Thus, for any r.v.\ $X$ with $\E X<\infty$, the terms $\E X$ and $\exp\E\ln X$ in \eqref{eq:} can be referred to, respectively, as the arithmetic and geometric means of the r.v.\ $X$. 
Since any bounded nonnegative r.v.\ can be approximated in distribution by uniformly bounded r.v.'s each taking finitely many nonnegative real values with equal probabilities, 
%such as the r.v.\ $X$ described in the beginning of this paragraph, 
the upper and lower bounds in \eqref{eq:} will each remain exact in an appropriate  sense if one considers only the r.v.'s with such discrete uniform distributions. In particular, one has the following immediate corollary from Theorem~\ref{th:} and Remark~\ref{rem:le}. 

\begin{corollary}\label{cor:a-g}
For any $n\in\N$, any $z=(z_1,\dots,z_n)\in\R^n$, and any function \break 
$f\colon\R\to\R$, let 
\begin{equation*}
\begin{gathered}
\ol z:=\tfrac1n\,(z_1+\dots+z_n), \quad 
z^\g:=\sqrt[\leftroot{3}\uproot{3}n]{|z_1\cdots z_n|}, \\ 
%(z-\ol z)^2:=\big((z_1-\ol z)^2,\dots,(z_n-\ol z)^2\big), \quad
%\sqrt z:=\big(\sqrt{z_1},\dots,\sqrt{z_n}\big), \\ 
z_{\max}:=z_1\vee\dots\vee z_n, \quad  
z_{\min}:=z_1\wedge\dots\wedge z_n, \quad %\\ 
f(z):=\big(f(z_1),\dots,f(z_n)\big).  
\end{gathered}
\end{equation*}
Then, for any real $V$, $E$, $F$ such that $0<V< E\wedge F$,  
\begin{align*}
%\sup\big\{\ol x-x^\g\colon x\in\R_+^n, n\in\N, \ol{(\sqrt x-\ol{\sqrt x})^2}\le V, 
%\ol{(\sqrt x-\sqrt{x_{\min}})^2}\le E\big\} 
%=&(2V)\vee E, \\ 
&%\begin{multlined}
\sup\big\{\ol x-x^\g\colon x=y^2,\; y\in\R_+^n,\; n\in\N,\; \ol{(y-\ol y)^2}\le V,\; 
\ol{(y-y_{\min})^2}\le E\big\} %\\ 
=(2V)\vee E, \\ 
%\end{multlined}\\ 
&%\begin{multlined}
\inf\big\{\ol x-x^\g\colon x=y^2,\; y\in\R_+^n,\; n\in\N,\; \ol{(y-\ol y)^2}\ge V,\; 
\ol{(y_{\max}-y)^2}\le F\big\} %\\ 
=(2V)\wedge E_{V,F}. 
%\end{multlined}
\end{align*}
\end{corollary}

%\newpage

The proof of Theorem~\ref{th:}, given in Section~\ref{proofs}, relies on the theory of Tchebycheff--Markoff systems. 
Major expositions of this theory and its applications are given in the monographs by Karlin and Studden \cite{karlin-studden} and Kre{\u\i}n and Nudel{\cprime}man \cite{krein-nudelman}. 
A brief review of the theory, which contains all the definitions and facts necessary for the proof in the present paper, is given in \cite{pinelis2011tchebycheff}. 
A condensed version of \cite{pinelis2011tchebycheff} can be found in \cite[Appendix~A]{radem-asymp}.

\section{Proofs}\label{proofs}

\begin{proof}[Proof of Proposition~\ref{prop:1}]\ 
%
%\noindent\textbf{(I)} 
Take any %nonnegative r.v.\ $X$ with $\E X<\infty$
$X\in\X_+$. Clearly, $E_X\ge V_X\ge0$. If $V_X=0$ then $\P(X=c)=1$ for some $c\in[0,\infty)$, whence $E_X=0$, so that $E_X=V_X=0$. If $V_X>0$ then $\E\sqrt X>\sqrt{m_X}$ and hence $E_X>V_X>0$. So, condition \eqref{eq:nonempty,hi} is necessary for $\X_{\hi;V,E}\ne\emptyset$. 
Vice versa, suppose now that \eqref{eq:nonempty,hi} holds. 
For any real $u$ and $v$ such that $0\le u<v$ and any $p\in[0,1]$, let $Y_{u,v,p}$ denote any  
r.v.\ such that 
\begin{equation}\label{eq:Y_{u,v,p}}
\text{$\P(Y_{u,v,p}=u)=p=1-\P(Y_{u,v,p}=v)$.}	
\end{equation}
If $E=V=0$ then $0\in\X_{\hi;V,E}$, and so, $\X_{\hi;V,E}\ne\emptyset$. If now $E>V>0$, let $X=Y_{u,v,p}^2$ with %$p\in[0,1]$ and any $u$ and $v$ such that $0\le u<v$ and 
\begin{equation}\label{eq:p,u,v,hi}
	p=\frac VE\quad\text{and any $u$ and $v$ such that\ \ $0\le u<v$\ \ and\ \ } v-u=\frac E{\sqrt{E-V}}. 
\end{equation}
Then $X\in\X_{\hi;V,E}$, and so, $\X_{\hi;V,E}\ne\emptyset$ in this case as well. 
Thus, the equivalence of the condition $\X_{\hi;V,E}\ne\emptyset$ and \eqref{eq:nonempty,hi} is checked. The equivalence of the condition $\X_{\lo;V,F}\ne\emptyset$ and \eqref{eq:nonempty,lo} is checked quite similarly; here, in the case when $F>V>0$, \eqref{eq:p,u,v,hi} is replaced by 
\begin{equation}\label{eq:p,u,v,lo}
	q:=1-p=\frac VF\quad\text{and any $u$ and $v$ such that\ \ $0\le u<v$\ \ and\ \ } v-u= \frac F{\sqrt{F-V}}. 
\end{equation}	
Thus, Proposition~\ref{prop:1} is proved. 
%part~(I) of Theorem~\ref{th:} is verified.
\end{proof}  

Before proceeding to the proofs of %part~(II) of Theorem~\ref{th:}
Propositions~\ref{prop:2} and \ref{prop:3}, let us state the following observation. 

\begin{lemma}\label{lem:var}
Take any r.v.\ $Z$ such that $\E Z=0$ and $\supp Z\subseteq[c,d]$ for some real $c$ and $d$. Then $c\le0\le d$, $\Var Z\le|c|d$, and $\Var Z=|c|d$ if and only if $\supp Z=\{|c|,d\}$. 
\end{lemma}

This follows immediately on noting that 
$c\le\E Z=0\le d$ and $\Var Z=\E Z^2=\E(Z-c)(Z-d)-cd\le-cd=|c|d$. 

Being very simple, Lemma~\ref{lem:var} seems to be a piece of common mathematical lore. 
E.g., the inequality $\Var Z\le|c|d$ in Lemma~\ref{lem:var} follows immediately from 
\cite[Lemma~2.2]{dharm-joag89}, by shifting and rescaling. 
In the case when $Z$ has a discrete distribution of the form given by \eqref{eq:discr}, Lemma~\ref{lem:var} was presented as Theorem~1 and second part of Proposition~1 in \cite{bhat-davis}. 

\begin{proof}[Proof of Proposition~\ref{prop:2}]\ 	
%
%\noindent\textbf{(II)} 
%
Suppose that $\X_{\lo;V,F}\ne\emptyset$ indeed, and take any $X\in\X_{\lo;V,F}$. Let $Y:=\sqrt X$, $a:=m_Y$, and $b:=M_Y$. 
%Using the above paragraph 
By Lemma~\ref{lem:var} with $Z:=Y-\E Y$, \break 
$c:=a-\E Y$, and $d:=b-\E Y$, 
\begin{equation}\label{eq:prop2}
	\begin{aligned}
	E_X=\E(Y-a)^2=\Var Y+(a-\E Y)^2\ge\Var Y+\frac{(\Var Y)^2}{(b-\E Y)^2}  
	=E_{V_X,F_X}=E_{V,F}  
\end{aligned}
\end{equation}
provided that $\infty>F>V$ -- 
with the inequality in \eqref{eq:prop2} turning into the equality if and only if $\supp Y=\{m_Y,M_Y\}$, that is, if and only if $\supp X=\{m_X,M_X\}$.  
This verifies %part~(II) of Theorem~\ref{th:} 
Proposition~\ref{prop:2} in the case when $\infty>F>V$. 

If now $F=V$ then, by %part~(I) of Theorem~\ref{th:}
Proposition~\ref{prop:1}, $F=V=0$. In this case, by the convention, $E_{V,F}=0$ and, on the other hand, 
for any $X\in\X_{\lo;V,F}$ one has $\supp X=c$ for some $c\in\R$, which implies $E_X=0$. 
So, %part~(II) of Theorem~\ref{th:}
Proposition~\ref{prop:2} holds as well in the case when $F=V$. 

Consider the remaining case, with $F=\infty$. Then, by the convention, $E_{V,F}=V$. For each $\vp\in(0,1)$, let $U_\vp$ be any r.v.\ whose distribution is \big(a mixture of a Bernoulli distribution and an exponential distribution\big) defined by the condition that 
\begin{equation}
	\E f(U_\vp)=(1-\vp)f(0)+(\vp-\vp^2)f(1)+\vp^2\int_0^\infty f(x)e^{-x}\dd x 
\end{equation}
for all nonnegative Borel functions $f$ on $\R$. 
Then $\E U_\vp=\vp=\Var U_\vp$ and $F_{U_\vp}=\infty$. 
Let now $X_\vp:=\frac V\vp\,U_\vp^2$. Then $X_\vp\in\X_{\lo;V,\infty}=\X_{\lo;V,F}$ and $E_{X_\vp}=(1+\vp)V$. So, 
\begin{equation}
	\inf\big\{E_X\colon X\in\X_{\lo;V,F}\big\}\le\inf\big\{(1+\vp)V\colon\vp\in(0,1)\big\}=V=E_{V,F}. 
\end{equation}
On the other hand, 
\begin{equation}\label{eq:E_X=}
E_X=V_X+(m_X-\E X)^2\ge V_X=V=E_{V,F}	
\end{equation}
for all $X\in\X_{\lo;V,F}$. 
Now \eqref{eq:E_{V,F}} follows as well in the case $F=\infty$. 
However, in this case the infimum in \eqref{eq:E_{V,F}} is not attained. 
Indeed, otherwise the inequality in \eqref{eq:E_X=} would for some $X\in\X_{\lo;V,F}$ turn into the equality, which would imply $\E X=m_X$ and hence $F_X=0$, which would contradict the assumption $F=\infty$. 
Thus, %part~(II) of Theorem~\ref{th:} 
Proposition~\ref{prop:2} is completely verified.  
\end{proof} 

\begin{proof}[Proof of Proposition~\ref{prop:3}]\ 
%\textbf{(IV)} 
The ``if'' side of %part (IV) of Theorem~\ref{th:} 
Proposition~\ref{prop:3} is quite straightforward to check. Let us verify the ``only if'' side. 
% of this part. 
Suppose that the inequalities in \eqref{eq:} turn simultaneously into the equalities, so that the upper and lower bound there are equal to each other, which is in turn equivalent to the statement that 
\begin{equation}\label{eq:equal}
E_X\le2V_X\le\frac{F_X V_X}{F_X-V_X}.  
\end{equation}
If $F_X=V_X$ then, by %part (IV) of Theorem~\ref{th:}
Proposition~\ref{prop:1}, $V_X=0$ and hence $\supp X=\{c\}$ for some $c\in[0,\infty)$, that is, the distribution of $\sqrt X$ is the (necessarily) symmetric distribution on the singleton set $\{\sqrt c\}\subset[0,\infty)$. 

It remains to consider the case $F_X>V_X$. Then the double inequality \eqref{eq:equal} can be rewritten as 
$2V_X\ge E_X\vee V_X$, which can be further rewritten as 
$$\Var Y\ge\max[(\E Y-a)^2,(b-\E Y)^2],$$ 
where $Y:=\sqrt X$, $a:=m_Y$, and $b:=M_Y$, so that $a\le\E Y\le b$. Therefore, 
\begin{equation}\label{eq:ge ge}
	2\Var Y\ge(\E Y-a)^2+(b-\E Y)^2\ge2(\E Y-a)(b-\E Y)\ge2\Var Y, 
\end{equation}
where the last inequality follows by Lemma~\ref{lem:var} (with $Z=Y-\E Y$). 
Hence, all the inequalities in \eqref{eq:ge ge} are actually the equalities. 
In particular, the equality \break $(\E Y-a)^2+(b-\E Y)^2=2(\E Y-a)(b-\E Y)$ implies $\E Y=(a+b)/2$. Also, again by Lemma~\ref{lem:var}, the equality $2(\E Y-a)(b-\E Y)=2\Var Y$ implies 
$\supp Y=\{a,b\}$. This, together with the condition $\E Y=(a+b)/2$, shows that the distribution of the r.v.\ $Y=\sqrt X$ is the symmetric distribution on the  set $\{a,b\}\subset[0,\infty)$. 
This completes the proof of %part (IV) of Theorem~\ref{th:}, as well as that of the entire theorem -- 
Proposition~\ref{prop:3}. 
\end{proof}

The proof of Theorem~\ref{th:} will be preceded by more notation and two lemmas. 
Take any $a$ and $b$ such that $0<a<b<\infty$ and 
introduce  
\begin{align*}
Q_{\hi;V,E}:=&\Big\{(\be_1,\be_2)\in(0,\infty)^2\colon 
	\be_2-\be_1^2=V,\,\be_2-2a\be_1+a^2=E\Big\},    \\ 	
Q_{\lo;V,F}:=&\Big\{(\be_1,\be_2)\in(0,\infty)^2\colon 
\be_2-\be_1^2=V,\,\be_2-2b\be_1+b^2=F\Big\},   	
\end{align*}
and then 
\begin{align}
	\Y_{\be_1,\be_2}:=&\big\{Y\in\X_+\colon\supp Y\subseteq[a,b],\,\E Y=\be_1,\,\E Y^2=\be_2\big\}, \label{eq:YY}\\
	S_{\be_1,\be_2}:=&\sup\big\{D_{Y^2}\colon Y\in\Y_{\be_1,\be_2}\big\}, \\ 
		I_{\be_1,\be_2}:=&\inf\big\{D_{Y^2}\colon Y\in\Y_{\be_1,\be_2}\big\} 
\end{align}
for $(\be_1,\be_2)\in(0,\infty)^2$, with the definition of $D_X$ in \eqref{eq:D_X} in mind; 
for brevity, the dependence on $a$ and $b$ is not made explicit in this notation.  

\begin{lemma}\label{lem:1}
Take any $(\be_1,\be_2)\in Q_{\hi;V,E}$ such that $\Y_{\be_1,\be_2}\ne\emptyset$. Then 
\begin{equation}
\text{$S_{\be_1,\be_2}\le(2V)\vee E$.}	
\end{equation}
\end{lemma}

\begin{lemma}\label{lem:2}
Take any $(\be_1,\be_2)\in Q_{\lo;V,F}$ such that $\Y_{\be_1,\be_2}\ne\emptyset$. Then 
\begin{equation}
\text{$I_{\be_1,\be_2}\ge(2V)\wedge E_{V,F}$.}	
\end{equation}
\end{lemma}

\begin{proof}[Proof of Lemma~\ref{lem:1}]
Note that 
\begin{align}
	S_{\be_1,\be_2}=\be_2-\exp\big(2I_{\ln;\be_1,\be_2}\big),\quad\text{where}\quad%\\ 
	I_{\ln;\be_1,\be_2}:=\inf\big\{\E\ln Y\colon Y\in\Y_{\be_1,\be_2}\big\}. \label{eq:I ln}
\end{align}
Using \cite[Proposition~1]{pinelis2011tchebycheff}, it is easy to see that that the sequence of functions \break 
$(1,\#,\#^2,\ln\#)$ is an $M_+$-system on $[a,b]$. 
Hence, by \cite[part~(II)(a) of Proposition~2]{pinelis2011tchebycheff} (with $n=2$), 
the infimum $I_{\ln;\be_1,\be_2}$ is attained at a r.v.\ of the form $Y=Y_{u,v,p}\in\Y_{\be_1,\be_2}$ with $0<u=a<v<\infty$ and $p\in[0,1]$, whose distribution is defined by \eqref{eq:Y_{u,v,p}}. 
These conditions on $Y_{u,v,p}%\in\Y_{\be_1,\be_2}
$, %$0<u=a<v<\infty$, and 
$u$, and $v$, together with the condition $(\be_1,\be_2)\in Q_{\hi;V,E}$, allow one to express 
$u$, $v$, $p$, $D_{Y_{u,v,p}^2}$, $V_{Y_{u,v,p}^2}$, and $E_{Y_{u,v,p}^2}$ uniquely in terms of $a$, $V$, and $E$, in accordance with \eqref{eq:p,u,v,hi}:
\begin{gather}
	u=a,\quad v=u+\frac E{\sqrt{E-V}},\quad p=\frac VE, \label{eq:u=a,...} \\ 
	 D_{Y_{u,v,p}^2}=p u^2 +q v^2 - u^{2 p} v^{2q}, \label{eq:D(u,v)} \\ 
	\quad V_{Y_{u,v,p}^2}=pq(v-u)^2=V,\quad E_{Y_{u,v,p}^2}=q(v-a)^2=q(v-u)^2=E,  
	\label{eq:D,V,E}
\end{gather}
where 
\begin{equation}
	q:=1-p. 
\end{equation}
%The latter equality in \eqref{eq:D,V,E} is obtained using the %crucial 
%condition $u=a$. 
It follows that %for $(\be_1,\be_2)\in Q_{\hi;V,E}$   
\begin{gather}
	S_{\be_1,\be_2}=\psi(0)\le\sup_{c\in[-u,\infty)}\psi(c),\quad\text{where} \label{eq:=sup psi} \\ 
	\psi(c):=D_{Y_{u,v,p}^2}=%\psi_{V,E}(c):=
	p(u+c)^2 +q(v+c)^2 - (u+c)^{2 p} (v+c)^{2q},  \label{eq:psi} 
\end{gather}
and $u$, $v$, $p$ are as in \eqref{eq:u=a,...}; 
cf.\ \eqref{eq:D(u,v)}. 
The supremum in \eqref{eq:=sup psi} is easy to find, and it depends only on $V$ and $E$. Indeed, %see reverse-jensen_scr.nb 
\begin{equation}
	\psi'''(c)
	=4 pq (p-q) (v-u)^3 (u+c)^{2p-3}(v+c)^{2q-3}
\end{equation}
equals $p-q$ in sign for all $c\in(-u,\infty)$. 
To find, for each $j\in\{0,1,2\}$, the limit $\psi^{(j)}(\infty-)$ of the derivative $\psi^{(j)}(c)$ as $c\to\infty$, for any $\ga\in\R$ write 
$(v+c)^\ga%\break 
=(u+c)^\ga(1+\vp)^\ga$, where $\vp:=\frac{v-u}{u+c}\sim\frac{v-u}c\to0$ and then write $$(1+\vp)^\ga=\sum_{i=0}^{2-j}\ga(\ga-1)\cdots(\ga-i+1)\frac{\vp^i}{i!}+o(c^{j-2}).$$  
Thus, one finds $\psi(\infty-)= 
2pq(u - v)^2=2V$ and $\psi'(\infty-)=\psi''(\infty-)=0$.  
% 
%At that, $\psi''(c)\to0$ and $\psi'(c)\to0$ as $c\to\infty$. 
Therefore and because $\psi'''$
equals $p-q$ in sign, one sees that $\psi'$ equals $p-q$ in sign, on the interval $(-u,\infty)$,  
which implies that the function $\psi$ is monotonic on the interval $[-u,\infty)$, with $\psi(-u)=q(u - v)^2=E$ and $\psi(\infty-)=2V$. 
Thus, the supremum in \eqref{eq:=sup psi} equals $(2V)\vee E$, which completes the proof of Lemma~\ref{lem:1}.  
\end{proof}

\begin{proof}[Proof of Lemma~\ref{lem:2}]
This proof is similar to that of Lemma~\ref{lem:1}. 
Here, instead of the infimum $I_{\ln;\be_1,\be_2}$ defined in \eqref{eq:I ln}, one deals with $S_{\ln;\be_1,\be_2}:=\sup\big\{\E\ln Y\colon Y\in\Y_{\be_1,\be_2}\big\}$. 
This supremum is attained at a r.v.\ of the form $Y=Y_{u,v,p}\in\Y_{\be_1,\be_2}$ with 
%$0<u<v=b<\infty$,  and $p\in[0,1]$, with 
\begin{gather}
	v=b,\quad u=v-\frac F{\sqrt{F-V}},\quad q=1-p=\frac VF,%\quad p=\frac{V-F}F; 
	\label{eq:v=b,...} \\  
	E_{Y_{u,v,p}^2}=q(b-u)^2=q(v-u)^2
	=\frac VF\,\Big(\frac F{\sqrt{F-V}}\Big)^2=\frac{VF}{F-V}=E_{V,F},  
	\label{eq:D,V,F}
\end{gather}
$D_{Y_{u,v,p}^2}$ as in \eqref{eq:D(u,v)}, $V_{Y_{u,v,p}^2}=%2pq(v-u)^2=
V$ as in \eqref{eq:D,V,E}, and $F_{Y_{u,v,p}^2}=p(v-u)^2=F$. % and $E_{Y_{u,v,p}^2}=q(v-u)^2=E$
The proof of Lemma~\ref{lem:2} is concluded with the observation that $\inf_{c\in[-u,\infty)}\psi(c)= (2V)\wedge E_{V,F}$ -- cf.\ the last sentence in the proof of Lemma~\ref{lem:1}. 
\end{proof}

\begin{proof}[Proof of Theorem~\ref{th:}]\ \\ 
Suppose that $\X_{\hi;V,E}\ne\emptyset$, so that condition \eqref{eq:nonempty,hi} holds. Both sides of \eqref{eq:sup} are obviously $0$ if $E=V=0$. 
To verify \eqref{eq:sup} in the remaining case $E>V>0$, 
fix any $X_*\in\X_{\hi;V,E}$. % such that 
Consider first the case 
\begin{equation}\label{eq:a,b,X_*}
	a:=\sqrt{m_{X_*}}>0\quad\text{and}\quad b:=\sqrt{M_{X_*}}<\infty. 
\end{equation} 

Letting now $Y_*:=\sqrt{X_*}$ and $(\be_1^*,\be_2^*):=(\E Y_*,\E Y_*^2)$, one has $(\be_1^*,\be_2^*)\in Q_{\hi;V,E}$ and $Y_*\in\Y_{\be_1^*,\be_2^*}$. Also, $D_{X_*}=D_{Y_*^2}\le S_{\be_1^*,\be_2^*}\le(2V)\vee E$, by Lemma~\ref{lem:1}. So,   
\begin{equation}\label{eq:le(2V)vee E}
	D_{X_*}\le(2V)\vee E, 
\end{equation}
%by Lemma~\ref{lem:1}, 
for any r.v.\ $X_*\in\X_{\hi;V,E}$ satisfying conditions \eqref{eq:a,b,X_*}. 

If now a r.v.\ $X_*\in\X_{\hi;V,E}$ is such that 
$m_{X_*}=0$, then $D_{X_*}\le\E X_*=E_{X_*}=E\le(2V)\vee E$, so that inequality \eqref{eq:le(2V)vee E} still holds. 

Take now any r.v.\ $X_*\in\X_{\hi;V,E}$ such that 
$m_{X_*}>0$ and $M_{X_*}=\infty$. Take then any $t\in(m_{X_*},\infty)$, and let $X_t:=X_*\wedge t$, so that $M_{X_t}\le t<\infty$, whence, by \eqref{eq:le(2V)vee E} with $X_t$ in place of $X_*$, one has $D_{X_t}\le(2V_{X_t})\vee E_{X_t}$. 
On the other hand, by dominated convergence with $t\to\infty$, one has 
$V_{X_t}\to V_{X_*}=V$, $E_{X_t}\to E_{X_*}=E$, %\break 
$\E X_t\to\E X_*$, and $\E\ln X_t\to\E\ln X_*$, and so, $D_{X_t}\to D_{X_*}$. 

Thus, inequality \eqref{eq:le(2V)vee E} holds for all $X_*\in\X_{\hi;V,E}$. That is, 
\begin{equation}
	S_{V,E}\le(2V)\vee E, 
\end{equation}
in the case $E>V>0$, where $S_{V,E}$ is as in \eqref{eq:sup}. 
On the other hand, again in the case $E>V>0$, for any $u,v,p$ as in \eqref{eq:p,u,v,hi} and any $c\in[-u,\infty)$, the r.v.\ 
$Y_{u+c,v+c,p}^2$ is in $\X_{\hi;V,E}$, and so, %by \eqref{eq:S_{V,E}}, 
\begin{equation}\label{eq:ge psi}
S_{V,E}\ge\sup_{c\in[-u,\infty)}\psi(c)=(2V)\vee E,  
\end{equation}
with $\psi(c)$ as in \eqref{eq:psi}. 
This concludes the proof of \eqref{eq:sup}. 

The proof of \eqref{eq:inf} is similar. 
Suppose that $\X_{\lo;V,F}\ne\emptyset$, so that condition \eqref{eq:nonempty,lo} holds. Both sides of \eqref{eq:inf} are obviously $0$ if $F=V=0$. 
Consider the remaining case $F>V>0$.  

Fix any $X_*\in\X_{\lo;V,F}$. Consider first the case when conditions \eqref{eq:a,b,X_*} hold. 

Letting now $Y_*:=\sqrt{X_*}$ and $(\be_1^*,\be_2^*):=(\E Y_*,\E Y_*^2)$, one has $(\be_1^*,\be_2^*)\in Q_{\lo;V,F}$ and $Y_*\in\Y_{\be_1^*,\be_2^*}$. 
Also, $D_{X_*}=D_{Y_*^2}\ge I_{\be_1^*,\be_2^*}\ge (2V)\wedge E_{V,F}$, by Lemma~\ref{lem:2}. So,   
\begin{equation}\label{eq:ge E wedge F}
	D_{X_*}\ge (2V)\wedge E_{V,F}, 
\end{equation}
for any r.v.\ $X_*\in\X_{\lo;V,F}$ satisfying conditions \eqref{eq:a,b,X_*}. 

Take now any $s$ and $t$ such that $0<s<t<\infty$ and let $X_{s,t}:=s\vee(t\wedge X_*)$, so that conditions \eqref{eq:a,b,X_*} be satisfied with $X_{s,t}$ in place of $X_*$. 
Hence, one will have $D_{X_{s,t}}\ge (2V_{X_{s,t}})\wedge E_{V_{X_{s,t}},F_{X_{s,t}}}$. Let now $s\downarrow0$ and $t\uparrow\infty$. Then $X_{s,t}\to X_*$ pointwise, $m_{X_{s,t}}\to m_{X_*}$, and $M_{X_{s,t}}\to M_{X_*}$. 
By dominated convergence, $\E X_{s,t}\to\E X_*$ and $V_{X_{s,t}}\to V_{X_*}=V$. 
If $F_{X_*}<\infty$, then $F_{X_{s,t}}\to F_{X_*}$, again by dominated convergence. 
If $F_{X_*}=\infty$, then clearly $F_{X_{s,t}}\le F_{X_*}$. 
Thus, in any case, $\limsup F_{X_{s,t}}\le F_{X_*}=F$. 
Moreover, by the Fatou lemma, $\E\ln X_*\le\liminf\E\ln X_{s,t}$, whence $D_{X_*}\ge\limsup D_{X_{s,t}}\ge\limsup\big[(2V_{X_{s,t}})\wedge E_{V_{X_{s,t}},F_{X_{s,t}}}\big]\ge(2V)\wedge E_{V,F}$, since $E_{V,F}$ is nonincreasing in $F$ and continuous in $(V,F)$ such that $F>V>0$. 

Thus, inequality \eqref{eq:ge E wedge F} holds for all $X_*\in\X_{\lo;V,F}$. That is, 
\begin{equation}
	I_{V,F}\ge (2V)\wedge E_{V,F}, 
\end{equation}
in the case $F>V>0$, where $I_{V,F}$ is as in \eqref{eq:inf}. 
On the other hand, again in the case $F>V>0$, for any $u,v,p$ as in \eqref{eq:p,u,v,lo} and any $c\in[-u,\infty)$, the r.v.\ 
$Y_{u+c,v+c,p}^2$ is in $\X_{\lo;V,F}$, and so, %by \eqref{eq:I_{V,F}}, 
\begin{equation}
I_{V,F}\le\inf_{c\in[-u,\infty)}\psi(c)=(2V)\wedge E_{V,F},  
\end{equation}
with $\psi(c)$ still as in \eqref{eq:psi}.  
This concludes the proof of \eqref{eq:inf}. 

Concerning the last sentence of Theorem~\ref{th:}, let $\X_{\hi,2;V,E}$ denote the set of all 
r.v.'s in $\X_{\hi;V,E}$ taking at most two values, and then let $S_{2;V,E}:=\sup\big\{D_X\colon X\in\X_{\hi,2;V,E}\big\}$. Suppose that $\X_{\hi;V,E}\ne\emptyset$, as is done in \eqref{eq:sup}, so that \eqref{eq:nonempty,hi} holds. 

If $E=V=0$, then $S_{V,E}=0$ and, on the other hand, $0\in\X_{\hi,2;V,E}$ and hence $0=D_0\le S_{2;V,E}\le S_{V,E}=0$, so that $S_{2;V,E}=S_{V,E}=(2V)\vee E$. 

Suppose now that $E>V>0$. Then for any $u,v,p$ as in \eqref{eq:p,u,v,hi} and any $c\in[-u,\infty)$ one has $Y_{u,v,p}^2\in\X_{\hi,2;V,E}$ and hence, by \eqref{eq:ge psi} and \eqref{eq:psi}, 
$(2V)\vee E=\sup_{c\in[-u,\infty)}\psi(c)=\sup_{c\in[-u,\infty)}D_{Y_{u,v,p}^2}\le S_{2;V,E}\le S_{V,E}=(2V)\vee E$, and so, the conclusion $S_{2;V,E}=S_{V,E}=(2V)\vee E$ holds. 

That is, the equality in \eqref{eq:sup} holds if the set $\X_{\hi;V,E}$ is replaced there by $\X_{\hi,2;V,E}$. The corresponding statement concerning the equality in \eqref{eq:inf} and the set $\X_{\lo;V,F}$ is verified quite similarly. 

Thus, Theorem~\ref{th:} is completely proved.   
\end{proof}

%\begin{center}
%	***
%\end{center}

%\bigskip
%\hrule
%\smallskip
%\hrule\hrule\hrule

%
%
%\bibliographystyle{srtnumbered}
\bibliographystyle{abbrv}
%%%%%\bibliographystyle{ims}
%%%%%\bibliography{are.citations}
%%%%%\bibliography{citat}
%%%%
%%%%%\bibliography{citations}
%%%%
\bibliography{C:/Users/ipinelis/Dropbox/mtu/bib_files/citations12.13.12}

\def\cprime{$'$} \def\polhk#1{\setbox0=\hbox{#1}{\ooalign{\hidewidth
  \lower1.5ex\hbox{`}\hidewidth\crcr\unhbox0}}}
  \def\polhk#1{\setbox0=\hbox{#1}{\ooalign{\hidewidth
  \lower1.5ex\hbox{`}\hidewidth\crcr\unhbox0}}}
  \def\polhk#1{\setbox0=\hbox{#1}{\ooalign{\hidewidth
  \lower1.5ex\hbox{`}\hidewidth\crcr\unhbox0}}} \def\cprime{$'$}
  \def\polhk#1{\setbox0=\hbox{#1}{\ooalign{\hidewidth
  \lower1.5ex\hbox{`}\hidewidth\crcr\unhbox0}}}
  \def\polhk#1{\setbox0=\hbox{#1}{\ooalign{\hidewidth
  \lower1.5ex\hbox{`}\hidewidth\crcr\unhbox0}}} \def\cprime{$'$}
\begin{thebibliography}{1}

\bibitem{bhat-davis}
R.~Bhatia and C.~Davis.
\newblock A better bound on the variance.
\newblock {\em Amer. Math. Monthly}, 107(4):353--357, 2000.

\bibitem{dharm-joag89}
S.~W. Dharmadhikari and K.~Joag-Dev.
\newblock Upper bounds for the variances of certain random variables.
\newblock {\em Comm. Statist. Theory Methods}, 18(9):3235--3247, 1989.

\bibitem{dragomir13}
S.~S. Dragomir.
\newblock Some reverses of the {J}ensen inequality with applications.
\newblock {\em Bull. Aust. Math. Soc.}, 87(2):177--194, 2013.

\bibitem{karlin-studden}
S.~Karlin and W.~J. Studden.
\newblock {\em Tchebycheff systems: {W}ith applications in analysis and
  statistics}.
\newblock Pure and Applied Mathematics, Vol. XV. Interscience Publishers John
  Wiley \& Sons, New York-London-Sydney, 1966.

\bibitem{11800}
O.~Klurman.
\newblock Problem 11800.
\newblock {\em Amer. Math. Monthly}, 121(8):739, 2014.

\bibitem{krein-nudelman}
M.~G. Kre{\u\i}n and A.~A. Nudel{\cprime}man.
\newblock {\em The {M}arkov moment problem and extremal problems}.
\newblock American Mathematical Society, Providence, R.I., 1977.
\newblock Ideas and problems of P. L. {\v{C}}eby{\v{s}}ev and A. A. Markov and
  their further development, Translated from the Russian by D. Louvish,
  Translations of Mathematical Monographs, Vol. 50.

\bibitem{pinelis2011tchebycheff}
I.~Pinelis.
\newblock Tchebycheff systems and extremal problems for generalized moments: a
  brief survey.
\newblock \url{http://arxiv.org/abs/1107.3493}, 2011.

\bibitem{radem-asymp}
I.~Pinelis.
\newblock An asymptotically {G}aussian bound on the {R}ademacher tails.
\newblock {\em Electron. J. Probab.}, 17:1--22, 2012.

\end{thebibliography}
%\bibliography{C:/Users/Iosif/Dropbox/mtu/bib_files/citations12.13.12}

\end{document}